\documentclass[a4paper,12pt]{article}
\usepackage[text={6.5in,9.5in},centering]{geometry}
\usepackage{graphicx,url,subfig}
\RequirePackage[OT1]{fontenc}
\RequirePackage{amsthm,amsmath,amssymb,amscd}
\RequirePackage[colorlinks,citecolor=blue,urlcolor=blue]{hyperref}
\usepackage{authblk}
\usepackage{etoolbox}
\usepackage{enumerate}
\usepackage{datetime}
\usepackage[normalem]{ulem} 
\usepackage{soul} 
\makeatother
\numberwithin{equation}{section}
\allowdisplaybreaks

\usepackage{ dsfont }
\usepackage{mathrsfs}
\usepackage[backend=biber,
	    style=numeric,
	    natbib=true,
            maxbibnames=99,
            sorting=nyt,
            abbreviate=true]{biblatex}
\renewbibmacro{in:}{%
  \ifentrytype{article}{}{\printtext{\bibstring{in}\intitlepunct}}}
\theoremstyle{plain}
\newtheorem{theorem}{Theorem}[section]

\newtheorem{lemma}[theorem]{Lemma}
\newtheorem{proposition}[theorem]{Proposition}

\theoremstyle{definition}

\newtheorem{assumption}[theorem]{Assumption}

\newtheorem{remark}[theorem]{Remark}

\newtheorem{example}[theorem]{Example}

\newcommand{\E}{\mathbb{E}}

\newcommand{\ud}{\ensuremath{\mathrm{d} }}

\newcommand{\calF}{\mathcal{F}}

\newcommand{\calN}{\mathcal{N}}

\newcommand{\calS}{\mathcal{S}}

\newcommand{\bbN}{\mathbb{N}}

\newcommand{\R}{\mathbb{R}}

\DeclareMathOperator{\Lip}{\mathit{L}}

\usepackage[strict=true,style=english]{csquotes}

\title{On weak convergence of stochastic wave equation with colored noise on $\R$}
\author{Wenxuan Tao\thanks{Email address: wxt399@student.bham.ac.uk}}
\affil{School of Mathematics, University of Birmingham,
Birmingham,
B15 2TT,
United Kingdom}

\date{}

\begin{document}
\maketitle
\begin{abstract}
In this paper, we study the following stochastic wave equation on the real line $\partial_t^2 u_{\alpha}=\partial_x^2 u_{\alpha}+b\left(u_\alpha\right)+\sigma\left(u_\alpha\right)\eta_{\alpha}$. The noise $\eta_\alpha$ is white in time and colored in space with a covariance structure $\E[\eta_\alpha(t,x)\eta_\alpha(s,y)]=\delta(t-s)f_\alpha(x-y)$ where $f_\alpha$ is continuous with respect to $\alpha$ in Fourier mode, see assumption~\ref{assumption on f}. We prove the continuity of the probability measure induced by the solution $u_\alpha$, in terms of $\alpha$, with respect to the convergence in law in the topology of continuous functions with uniform metric on compact sets. We also give several examples of $f_\alpha$ such that our theorem applies to.
\end{abstract}
\vspace{10mm}

        \noindent{\it \noindent }MSC 2020: 60B10; 60H15

%

\setlength{\parindent}{1.5em}


\section{Introduction}
In this paper, we consider the stochastic wave equation on $[0,\infty)\times \R$ driven by Gaussian noise $\eta_\alpha$ which is white in time and colored in space 
\begin{equation}\label{colored SWE}
    \left\{
        \begin{aligned}
            &\frac{\partial^2 u_{\alpha}}{\partial t^2}(t,x)=\frac{\partial^2 u_{\alpha}}{\partial x^2}(t,x)+b\left(u_\alpha(t,x)\right)+\sigma\left(u_\alpha(t,x)\right)\eta_{\alpha}\,,\\
            &u_\alpha(0,x)=u_0(x)\,,\quad x\in \R\,,\\
            &\frac{\partial u_\alpha}{\partial t}(0,x)=v_0(x)\,,\quad x\in \R\,,
        \end{aligned}
    \right.
\end{equation}
where $b$ and $\sigma$ are Lipschitz functions with Lipschitz constant $\Lip_b$ and $\Lip_\sigma$, respectively. The initial value $u_0$ is assumed to be bounded and Lipschitz. And $v_0$ is assumed to be measurable and bounded.
The colored noise $\eta_\alpha$ has the covariance structure
\begin{equation}\label{covariance}
    \E[\eta_\alpha(t,x)\eta_\alpha(s,y)]=\delta(t-s)f_\alpha(x-y)\,,
\end{equation}
where $f_\alpha$ is a covariance function satisfying Assumption~\ref{assumption on f} with $\alpha\in I\subset\R$ for certain subset of $\R$.
\begin{remark}
    If $f_\alpha(x)$ is given by the Dirac delta function $\delta(x)$, the equation is 
 driven by space-time white noise.
\end{remark}
We denote $\mathcal{F}f(\xi)=\hat{f}(\xi)$  the Fourier transform of $f$, in the following sense (see for example \cite[Section 2]{Dalang1999EJP}). For any function $f\in L^1(\R)$, we define the Fourier transform of $\phi$ as
\begin{equation}\label{fourier transform}
    \calF(f)(\xi):=\int_\R f(x)e^{- i x\xi}\ud x\,.
\end{equation}
Then, the Fourier transform can be extended to the space of tempered distribution via duality. The Fourier transform $\hat{f_\alpha}$ is the function such that the following holds, for all functions $\phi$ in the Schwartz space $\calS(\R)$
\begin{align*}
    \int_\R \phi(x)\overline{\hat{f_\alpha}(x)}\ud x= \frac{1}{2\pi}\int_\R \hat{\phi}(\xi)\overline{f_\alpha(\xi)}\ud \xi\,,
\end{align*}
where $\overline{z}$ denotes the complex conjugate of $z$.
We may easily see from this relation that for any $\phi, \psi \in \calS(\R)$,
\begin{equation}\label{fourier mode schwartz}
        \int_\R\int_\R \phi(x)f_\alpha(x-y)\psi(y)\ud x\ud y=\int_\R f_\alpha(x)(\phi\ast \Tilde{\psi})(x)\ud x=\frac{1}{2\pi}\int_\R \hat{f_\alpha}(\xi)\hat{\phi}(\xi)\overline{\hat{\psi}(\xi)}\ud \xi\,,
\end{equation}
where $\tilde{\psi}(x)=\psi (-x)$.
 The covariance functions $\{f_\alpha\}$ are naturally even and non-negative. Now we make the following additional assumptions on $f_\alpha$.
\begin{assumption}\label{assumption on f}
    \begin{enumerate}
        \item For some $\alpha_0\in \overline{I}$, $\hat{f}_\alpha(\xi)\to \hat{f}_{\alpha_0}(\xi)$ in $L^1_{loc}(\R)$ as $\alpha\to \alpha_0$, where $\overline{I}$ denotes the closure of $I$. Here if $\alpha_0\in\overline{I}\setminus I$, we assume the limit exists and denote it by $\hat{f}_{\alpha_0}$.
        \item 
        There exists $\omega>0$ such that $\hat{f}_\alpha(\xi)$ is uniformly bounded in $\alpha$ and converges to $\hat{f}_{\alpha_0}(\xi)$ point wise on $\R\setminus[-\omega,\omega]$, as $\alpha$ goes to $\alpha_0$.
        \item For each $\alpha\in I\bigcup\{\alpha_0\}$, $f_\alpha=h_\alpha\ast h_\alpha$ for certain real function $h_\alpha$. As a result, $\hat{f}_\alpha=(\hat{h}_\alpha)^2$.
        \item We also assume that there exists $0<\beta<1$ such that
        \begin{equation}\label{improved dalang's condition}
            \sup_{\alpha\in I\bigcup\{\alpha_0\}}\int_\R \frac{\hat{f_\alpha}(\xi)}{(1+|\xi|^2)^{1-\beta}}\ud \xi<\infty\,.
        \end{equation}
    \end{enumerate}
\end{assumption}
\begin{remark}
    \begin{itemize}
        \item Assumptions 1 and 2 indicate $f_\alpha$ converges to $f$ in the sense of Schwartz distribution. And the same convergence holds for $\hat{h}_\alpha=\sqrt{\hat{f}_\alpha}\to \hat{h}_{\alpha_0}$. 
        \item Assumption 3 can be easily fulfilled by any classical kernels which satisfy semigroup property.
        \item Assumption 4 is an improved Dalang's condition \cite[Section 4.4.4]{dalang2011expositiones}, which ensures the well-posedness and H\"{o}lder regularity of the mild solution (see the definition in Section 2).
    \end{itemize}
\end{remark}
\begin{example}\label{examples}
Here we list some examples of the covariance functions $f_\alpha$. For the sake of completeness, an explicit justification of how these examples fit into our assumptions is left in Section~\ref{Validation of the examples}.
    \begin{enumerate}
        \item Riesz kernels (see e.g. \cite[Example 1]{Dalang1999EJP})
        
        Let $g_\alpha(x)=\frac{1}{|x|^\alpha}$ for $\alpha\in (0,1)$ and define $f_\alpha$ to be the Riesz kernel
        \begin{align}\label{Riesz kernels}
            f_\alpha(x)=c_{1-\alpha}g_\alpha(x),
        \end{align}
        where $c_\alpha=\frac{\sin\left(\frac{\alpha\pi}{2}\right)\Gamma(1-\alpha)}{\pi}$. 
        \item Heat kernels
        
        For $\alpha>0$, define
        \begin{align}\label{Heat kernels}
            f_\alpha(x)=\frac{1}{\sqrt{2\pi\alpha}}e^{-\frac{|x|^2}{2\alpha}}\,.
        \end{align}
        \item Bessel kernels
        
        For $\alpha>0$, define
        \begin{align}\label{Bessel kernels}
            f_\alpha(x)=\frac{1}{(4\pi)^{\alpha/2}\Gamma(\alpha/2)}\int_0^\infty \frac{e^{-\frac{\pi|x|^2}{y}}-\frac{y}{4\pi}}{y^{1+\frac{1-\alpha}{2}}}\ud y\,.
        \end{align}
        \item For any $\rho\in L^1(\R)\bigcap L^2(\R)$ with $\|\rho\|_{L^1(\R)}=1$, define $\rho_\alpha(x)=\frac{1}{\alpha}\rho(\frac{x}{\alpha})$ for $\alpha>0$ and
        \begin{align}\label{general example}
            f_\alpha=\rho_\alpha\ast \rho_\alpha\,.
        \end{align}
    \end{enumerate}
\end{example}

A rigorous definition of the noise $\eta_\alpha$ is introduced in Section~\ref{preliminaries}.
In this paper, we study the solution $u_\alpha$ of \eqref{colored SWE} as a function of $\alpha$ and prove the following main result.
\begin{theorem}\label{main theorem}
    Let $\mathcal{C}=C\left(\left[0,T\right]\times\R\right)$ denote the continuous functions on $[0,T]\times \R$, endowed with the supremum norm on compact set. Then for any sequence $\{\alpha_n\}_{n=1}^\infty\subset I$  such that $\lim_{n\xrightarrow{}\infty}\alpha_n=\alpha_0$, $u_{\alpha_n}$ converges to $u_{\alpha_0}$ in law in the space $\mathcal{C}$ as $n\xrightarrow{} \infty$.
\end{theorem}
\begin{remark}\label{key idea}
    We can see the key idea of this paper from the first example. Traditionally, in the space domain, $f_\alpha$ given by \eqref{Riesz kernels} converges point wise to $0$ as $\alpha\to 1$. But in the frequency domain, $\hat{f}_\alpha(\xi)=1/|\xi|^{1-\alpha}$ converges to the Fourier transform of Dirac delta function $\hat{\delta}\equiv 1$. See Section~\ref{Validation of the examples} for more details.
\end{remark}
Recently, researchers have begun investigating similar convergence in law properties in the context of stochastic partial differential equations (SPDEs). For instance, a weak convergence result in the setting of the stochastic heat equation was studied in \cite{BEZDEK2016SPA} by Pavel Bezdek. Giordano et al. \cite{Giordano2020Bernoulli} explored the weak convergence with respect to the Hurst index of both stochastic heat and wave equations with fractional additive noise $\sigma(x)=1$. One key point they have proven is the continuity of the solution operator in the heat equation case with bounded drift $b$. The authors of \cite{giordano2020SPA} addressed the same kind of results in Parabolic and Hyperbolic Anderson Models $\sigma(x)=x$, utilizing the method of chaos expansion. 
In the most recent paper \cite{balan2023continuity}, Balan and Liang studied Parabolic and Hyperbolic Anderson Models driven by colored noise in both space and time, where the solution is understood in Skorohod sense. They discussed two cases. In the regular case, the spatial covariance is given by the Riesz kernel. In the rough case, the spatial noise is given by a fractional Brownian motion with Hurst index $0<H<\frac{1}{2}$. We establish our result following the approach of Jolis and Viles \cite{Jolis2007ContinuityIL}, who combine the finite dimensional convergence and tightness to prove the convergence in law of the local time of fractional Brownian motion with respect to the Hurst parameter. Additionally, we build upon the work of Wu and Xiao \cite{wu2009continuity}, who extended these results to higher dimensions and the case of multiple independent copies of fractional Brownian motion.

The novelty of this paper lies in two key aspects. First, we consider a generalized diffusion factor 
$\sigma(u)$, extending beyond the additive noise studied in \cite{Giordano2020Bernoulli} and the multiplicative noise in \cite{giordano2020SPA}. Second, and more importantly, we generalize the covariance function $f_\alpha$. In our framework, the noise is no longer restricted to being generated by fractional Brownian motion or structured according to the Riesz kernel as discussed before. Instead, these cases are encompassed as special instances, as illustrated in Example~\ref{examples}. Moreover, such generalization on $f_\alpha$ reveals the key point of the main theorem. See Remark~\ref{key idea}.

The proof of the main theorem develops as follows. In Section~\ref{preliminaries}, we introduce the rigorous definition of the noises and the method of coupling. We also provide the existence and uniqueness of solutions to \eqref{colored SWE} based on the result from \cite{dalang2011expositiones} as well as the idea of mild solution. In the next section~\ref{Holder regularity} we prove the H\"{o}lder continuity of the process $u_\alpha(t,x)$. In Section~\ref{Finite dimensional convergence}, we prove the convergence of finite dimensional distributions of $u_\alpha$ to that of $u$.  Finally, we combine these results to derive the main theorem~\ref{main theorem}. We end this paper in Section~\ref{Validation of the examples} by verifying the examples given in~\ref{examples} with respect to the assumptions~\ref{assumption on f}. Throughout this paper, $C$ will denote a generic constant which may change from line to line. We fix a positive number $T$ as the maximal time being considered.
\section{Preliminaries}\label{preliminaries}
\subsection{ A rigorous definition of the noises}\label{definition of the noise}

To be precise, $\eta_\alpha$ and the white noise $\eta$ defined rigorously as follows. Given a complete probability space $(\Omega, \mathcal{F},\mathbb{P})$, the set of noises ${\eta_\alpha}$ is defined from a family of centered Gaussian random functionals $B_\alpha$ defined on $\mathcal{D}$, where $\mathcal{D}=C_c^\infty([0,T]\times \R)$ refers to the space of infinitely differentiable functions with compact support. For any $\phi,\psi \in \mathcal{D}$, the covariance
\begin{align*}
    &\E\left[B_\alpha(\phi)B_\alpha(\psi)\right]=\int_0^T\int_\R \int_\R \phi(s,x)\psi(s,y)f_\alpha(x-y)\ud y\ud x\ud s\,,\\
    &\E\left[B(\phi)B(\psi)\right]=\int_0^T \int_\R \phi(s,x)\psi(s,x)\ud x\ud s\,.
\end{align*}
Now, with an abuse of notation with suitable approximation arguments,, we define for any bounded Borel set $A$,
\begin{align*}
    &B_\alpha([0,t]\times A):=B_\alpha(\mathds{1}_{[0,t]\times A})\,,\\
    &B([0,t]\times A):=B(\mathds{1}_{[0,t]\times A})\,,\\
\end{align*}
and another abuse of notation,
\begin{align*}
    B_\alpha(t,x):=\left\{
    \begin{aligned}
        &B_\alpha(\mathds{1}_{[0,t]\times [0,x]}),\quad\quad&&\text{for}\, x\geq0\,,\\
        &B_\alpha(-\mathds{1}_{[0,t]\times [x,0]}),\quad&&\text{for}\, x<0\,,
    \end{aligned}
    \right.
\end{align*}
and
\begin{align*}
    B(t,x):=\left\{
    \begin{aligned}
        &B(\mathds{1}_{[0,t]\times [0,x]}),\quad\quad&&\text{for}\, x\geq0\,,\\
        &B(-\mathds{1}_{[0,t]\times [x,0]}),\quad&&\text{for}\, x<0\,.
    \end{aligned}
    \right.
\end{align*}
The colored noise $\eta_\alpha$ and white noise $\eta$ are then defined as
\begin{equation}\label{B and eta}
    \begin{aligned}
        &\eta_\alpha(t,x)=\frac{\partial^2 B_\alpha}{\partial t\partial x}(t,x)\,,\\
    &\eta(t,x)=\frac{\partial^2 B}{\partial t\partial x}(t,x)\,.
    \end{aligned}
\end{equation}
Let $G_t(x)$ denote the wave kernel
\begin{equation}\label{wave kernel}
    G_t(x)=\frac{1}{2}\mathds{1}_{(|x|<t)}\,,
\end{equation}
and let $I_0$ denote the deterministic linear wave solution,
\begin{equation}
    I_0(t,x)=\frac{1}{2}\int_{x-t}^{x+t}v_0(y)\ud y+\frac{1}{2}\left(u_0(x+t)+u_0(x-t)\right)\,.
\end{equation}
The Fourier transform of spatial covariance function $\hat{f}_\alpha$ defines a measure $\mu_\alpha(\ud \xi)=\hat{f}_\alpha(\xi)\ud \xi$, which is called the spectrum measure. As a result of \eqref{improved dalang's condition},
we can check the Dalang's condition \cite[Section 4.4.4]{dalang2011expositiones}
\begin{equation}
    \int_\R \frac{\mu_\alpha(\ud \xi)}{1+\xi^2}=\int_\R \frac{\hat{f_\alpha}(\xi)}{(1+|\xi|^2)^{1-\beta}}\frac{1}{(1+|\xi|^2)^\beta}\ud \xi<\infty\,,
\end{equation}
which ensures the existence and uniqueness of adapted mild random field solutions in the 
 following form with finite second moment 
\begin{equation}\label{mild solution}
    \begin{aligned}
        &u_\alpha(t,x)=I_0(t,x)+\int_0^t\int_\R G_{t-s}(x-y)b(u_\alpha(s,y))\ud y\ud s\\
        &\qquad\qquad+\int_0^t\int_\R G_{t-s}(x-y)\sigma(u_\alpha(s,y))\eta_\alpha(\ud s,\ud y)\,,\\
    \end{aligned}
\end{equation}
where the stochastic integral is in the sense of Dalang-Walsh \cite[Section 2]{Dalang1999EJP}.
\subsection{The method of coupling}
The method of coupling helps us to set the whole family of colored noises $\{\eta_\alpha\}$ in a single probability space $(\Omega,\calF, \mathbb{P})$ which is also the probability space in which the white noise $\eta$ is defined. By coupling, we can define the colored noises $B_\alpha$ as follows: for any bounded Borel set $A\subset \R$
\begin{equation}\label{coupling}
    B_\alpha([0,t]\times A)=\int_0^t\int_\R(\mathds{1}_A\ast h_\alpha)(x)B(\ud s,\ud x)\,.
\end{equation}
We can check that the resulting $B_\alpha$ is exactly the one we defined in \eqref{covariance} since
\begin{equation}
    f_\alpha(x)=(h_\alpha\ast h_\alpha)(x)\,.
\end{equation}
From the construction of Dalang-Walsh integral, the coupling relation \eqref{coupling} implies that for any $t\in [0,T]$ and any adapted random field $X(t,x)$ such that
\begin{equation}\label{coupling condition}
    \E\int_0^t\int_\R\int_\R |X(s,x)X(s,y)|f_\alpha(x-y)\ud x\ud y\ud s<\infty,
\end{equation}
we have
\begin{equation}\label{coupling result}
    \int_0^t\int_\R X(s,y)\eta_\alpha(\ud s,\ud y)=\int_0^t\int_\R (X(s,\cdot)\ast h_\alpha)(y)\eta(\ud s,\ud y)\,.
\end{equation}

We also refer the reader to \cite[Section 2.2]{giordano2020SPA} for more details of coupling.

{\section{Proof of the main theorem}}
\subsection{H\"{o}lder continuity}\label{Holder regularity}
For any random variable $u$, we denote
\begin{align*}
    \|u\|_k:=\|u\|_{L^k(\Omega)}=(\E|u|^k)^{\frac{1}{k}}\,.
\end{align*}
And a space time norm:
\begin{align*}
    \calN_{\gamma,k}(u)=\sup_{s\in [0,T]}\sup_{y\in \R}\left(e^{-\gamma s}\|u(s,y)\|_{L^k(\Omega)}\right)\,,\quad \gamma>1,k\geq 2\,.
\end{align*}
We introduce some useful lemmas. The first lemma is a extension of the relation \eqref{fourier mode schwartz} to more general functions.
\begin{lemma}\label{fourier mode}
    If $f$ is lower semicontinuous, then for all Borel finite measures $\mu$ on $\R$,
    \begin{equation}
        \int_\R\int_\R f(x-y)\mu(\ud x)\mu(\ud y)=\frac{1}{2\pi}\int_\R|\hat{\mu}(\xi)|^2\hat{f}(\xi)\ud \xi\,.
    \end{equation}
\end{lemma}
This lemma is a modification of the Corollary 3.4 of \cite{foondun2013TAMS}, where the measure $\mu$ is assumed to be a Borel probability measure. In fact, for any Borel finite measure $\mu$, we can separate it as $\mu=\mu_+-\mu_-$, where both $\mu_+$ and $\mu_-$ are nonnegative and finite measures. By normalization to a probability measure, Lemma~\ref{fourier mode} follows.
First we prove the following uniform boundedness result.

\begin{proposition}\label{uniform boundedness}
Let $u_\alpha(t,x)$ be the solution to \eqref{colored SWE}, then
    \begin{equation}\label{alpha uniform bound}
    \sup_{\alpha\in I\bigcup\{\alpha_0\}}\sup_{t\in [0,T]}\sup_{x\in \R}\|u_\alpha(t,x)\|_k<\infty\,,\quad \forall k\geq 2\,.
    \end{equation}
\end{proposition}
\begin{proof}
    For any fixed $k\geq 2$, we apply the $k$-th norm $\|\cdot\|_{L^k(\Omega)}$ on the both sides of mild solution representation \eqref{mild solution}, and take square
    \begin{align*}
        \|u_\alpha(t,x)\|_k^2\leq & 3\|I_0(t,x)\|_k^2+3\left\|\int_0^t\int_\R G_{t-s}(x-y)b(u_\alpha(s,y))\ud y \ud s\right\|_k^2\\
        &+3\left\|\int_0^t\int_\R G_{t-s}(x-y)\sigma(u_\alpha(s,y))\eta_\alpha(\ud s,\ud y)\right\|_k^2\,.
    \end{align*}
    By assumption, $u_0$ and $v_0$ are bounded, thus
    \begin{align*}
        \|I_0(t,x)\|_k^2\leq \left(T\sup_{y\in \R}|v_0(y)|+\sup_{y\in \R}|u_0(u)|\right)^2=C_0< \infty.
    \end{align*}
    Using Minkowski inequality and the Lipschitz condition of $b$,
    \begin{align*}
        &e^{-2\gamma t}\left\|\int_0^t\int_\R G_{t-s}(x-y)b(u_\alpha(s,y))\ud y \ud s\right\|_k^2\\
        &\quad\leq \left(\int_0^t \left(|b(0)|+\Lip_be^{-\gamma (t-s)}\sup_{s\in [0,T]}\sup_{y\in \R}e^{-\gamma s}\|u_\alpha(s,y)\|_k\right)\int_\R G_{t-s}(x-y)\ud y\ud s\right)^2\\
        &\quad\leq \left(C+C\calN_{\gamma,k}(u_\alpha)\int_0^t(t-s)e^{-\gamma(t-s)}\ud s\right)^2\\
        &\quad \leq C_1+\frac{C_1}{\gamma^4}\calN_{\gamma,k}(u_\alpha)^2\,.
    \end{align*}
    An application of Burkholder's inequality, Minkowski inequality and Lemma~\ref{fourier mode} implies
    \begin{align*}
        &e^{-2\gamma t}\left\|\int_0^t\int_\R G_{t-s}(x-y)\sigma(u_\alpha(s,y))\eta_\alpha(\ud s,\ud y)\right\|_k^2\\
        &\quad \leq C_k\int_0^te^{-2\gamma(t-s)}\Bigg(\int_\R\int_\R G_{t-s}(x-y)e^{-\gamma s}\|\sigma(u_\alpha(s,y))\|_k f_\alpha(y-z)\\
        &\qquad \qquad \qquad \qquad \qquad \quad \times G_{t-s}(x-z)e^{-\gamma s}\|\sigma(u_\alpha(s,z))\|_k \ud y\ud z\Bigg)\ud s\\
        &\quad\leq C(1+\calN_{\gamma,k}(u_\alpha)^2)\int_0^t e^{-2\gamma(t-s)}\left(\int_\R \hat{f}_\alpha(\xi)\left|\calF G_{t-s}(x-\cdot)(\xi)\right|^2\ud \xi\right)\ud s\\
        &\quad\leq C(1+\calN_{\gamma,k}(u_\alpha)^2)\int_0^t e^{-2\gamma (t-s)} (t-s)^2\left(\sup_\alpha \int_\R \frac{\hat{f}_\alpha(\xi)}{1+|\xi|^2}\ud \xi\right)\ud s\\
        &\quad\leq C_2+\frac{C_2}{\gamma^3}\calN_{\gamma,k}(u_\alpha)^2\,.
    \end{align*}
    Choose $\gamma=\gamma_0=\max\{(12C_1)^{\frac{1}{4}},(12C_2)^{\frac{1}{3}},2\}$, we have
    \begin{align*}
        e^{-2\gamma t}\|u_\alpha(t,x)\|_k^2\leq 3C_0+3C_1+3C_2+\frac{1}{2}\calN_{\gamma,k}(u)^2\,.
    \end{align*}
    Take supremum over $t\in [0,T]$ and $x\in \R$,
    \begin{align*}
        \calN_{\gamma,k}(u_\alpha)^2\leq C\,.
    \end{align*}
    Notice that the factor $e^{-\gamma t}$ is bounded below by $e^{-\gamma_0 T}$,
    \begin{align*}
        \sup_{t\in [0,T]}\sup_{x\in \R}\|u_\alpha(t,x)\|_k\leq Ce^{\gamma_0 T}<\infty\,.
    \end{align*}
    Take supremum over $\alpha\in I$ and we get the desired result.
\end{proof}
Now we are prepared to show the H\"{o}lder regularity. We claim the following result.
\begin{proposition}\label{holder}
For any $\alpha\in I\bigcup\{\alpha_0\}$, let $u_\alpha$ be the solution to \eqref{colored SWE}, then for any $x,x'\in\R$ and $t,t'\in [0,T]$,
    \begin{equation}
        \left\|u_\alpha(t',x')-u_\alpha(t,x)\right\|_k\leq C_T\left(|x'-x|^{\beta}+|t'-t|^\beta\right)\,.
    \end{equation}
Thus, an application of Kolmogorov continuity theorem implies that for each $\alpha$, there exists a continuous modification of $u_\alpha(t,x)$ with respect to $(t,x)\in [0,T]\times\R$. 
\end{proposition}
\begin{proof}
Step I. we prove the H\"{o}lder regularity in the spatial variable
\begin{align*}
    &\left\|u_\alpha(t,x)-u_\alpha(t,x')\right\|_k\\
    &\quad\leq |I_0(t,x)-I_0(t.x')|\\
    &\quad\quad +\left\|\int_0^t\int_\R \left(G_{t-s}(x-y)-G_{t-s}(x'-y)\right)b(u_\alpha(s,y))\ud y\ud s\right\|_k\\
    &\quad\quad+\left\|\int_0^t\int_\R \left(G_{t-s}(x-y)-G_{t-s}(x'-y)\right)\sigma(u_\alpha(s,y))\eta_\alpha(\ud s,\ud y)\right\|_k\\
    &\quad=J_0+J_1+J_2\,.
\end{align*}
Since $u_0$ is Lipschitz with Lipschitz constant $L_{u_0}$ and $v_0$ is bounded,
\begin{align*}
    J_0\leq 2\sup_{y\in \R}|v_0(y)||x-x'|+\Lip_{u_0}|x-x'|\,.
\end{align*}
For $J_1$, we apply Lemma~\ref{spatial holder 1} to obtain
\begin{align*}
    J_1=&\left\|\int_0^t\int_\R \left(G_{t-s}(x-y)-G_{t-s}(x'-y)\right)b(u_\alpha(s,y))\ud y\ud s\right\|_k\\
    &\quad\leq \int_0^t\int_\R \left|G_{t-s}(x-y)-G_{t-s}(x'-y)\right|\left\|b(u_\alpha(s,y))\right\|_k\ud y\ud s\\
    &\quad \leq\int_0^t\int_\R \left|G_{t-s}(x-y)-G_{t-s}(x'-y)\right|\left(|b(0)|+\Lip_b \left\|u_\alpha(s,y)\right\|_k\right)\ud y\ud s\\
    &\quad \leq \left(|b(0)|+\Lip_b \sup_{\alpha\in (\alpha_0,1)}\sup_{s\in [0,t]}\sup_{y\in \R}\left\|u_\alpha(s,y)\right\|_k\right)\int_0^t\int_\R \left|G_s(x-y)-G_s(x'-y)\right|\ud y\ud s\\
    &\quad\leq C(|x'-x|\wedge 1)\,.
\end{align*}
For $J_2$, we apply Lemma~\ref{spatial holder},
\begin{align*}
    J_2&=\left\|\int_0^t\int_\R \left(G_{t-s}(x-y)-G_{t-s}(x'-y)\right)\sigma(u_\alpha(s,y))\eta_\alpha(\ud s,\ud y)\right\|_k\\
    &\quad\leq C_k\left(|\sigma(0)|+\Lip_\sigma \sup_{\alpha\in (\alpha_0,1)}\sup_{s\in [0,t]}\sup_{y\in \R}\left\|u_\alpha(s,y)\right\|_k\right)\\
    &\quad\quad \times\left|\int_0^t\int_\R\int_\R \left|G_{s}(x-y)-G_s(x'-y)\right|f_\alpha(y-z)\left|G_{s}(x-z)-G_s(x'-z)\right|\ud y \ud z \ud s\right|^{\frac{1}{2}}\\
    &\quad \leq C(|x'-x|\wedge 1)^{\beta}\,.
\end{align*}
As a result,
\begin{equation}
    \left\|u_\alpha(t,x)-u_\alpha(t,x')\right\|_k\leq C(|x'-x|\wedge 1)^{\beta}\leq C|x'-x|^{\beta}\,.
\end{equation}
Step II. we prove the H\"{o}lder regularity in the time variable. Assume that $0\leq t< t'\leq T$, we have
\begin{align*}
    &\left\|u_\alpha(t',x)-u_\alpha(t,x)\right\|_k\\
    &\leq  |I_0(t',x)-I_0(t.x)|\\
    &+\left\|\int_0^{t'}\int_\R G_{t'-s}(x-y)b(u_\alpha(s,y))\ud y \ud s-\int_0^t\int_\R G_{t-s}(x-y)b(u_\alpha(s,y))\ud y\ud s\right\|_k\\
    & +\left\|\int_0^{t'}\int_\R G_{t'-s}(x-y)\sigma(u_\alpha(s,y))\eta_\alpha( \ud s,\ud y)-\int_0^t\int_\R G_{t-s}(x-y)\sigma(u_\alpha(s,y))\eta_\alpha(\ud s,\ud y)\right\|_k\\
    &\quad =H_0+H_1+H_2\,.
\end{align*}
Since $u_0$ is Lipschitz and $v_0$ is bounded,
\begin{align*}
    H_0\leq \sup_{y\in \R}|v_0(y)||t-t'|+\Lip_{u_0}|t-t'|\,.
\end{align*}
For $H_1$, we apply Lemma~\ref{time holder 1} to obtain
\begin{align*}
    &H_1\leq \left\|\int_0^{t}\int_\R \left(G_{t'-s}(x-y)-G_{t-s}(x-y\right)b(u_\alpha(s,y))\ud y \ud s\right\|_k\\
    &\quad\quad+\left\|\int_t^{t'}\int_\R G_{t'-s}(x-y)b(u_\alpha(s,y))\ud y\ud s\right\|_k\\
    &\quad =H_1^{(1)}+H_1^{(2)}\,.
\end{align*}
Here,
\begin{align*}
    H_1^{(1)}\leq& \left(|b(0)|+\Lip_b \sup_{\alpha\in (\alpha_0,1)}\sup_{s\in [0,t]}\sup_{y\in \R}\left\|u_\alpha(s,y)\right\|_k\right)\int_0^t\int_\R \left|G_{t'-s}(x-y)-G_{t-s}(x-y)\right|\ud y \ud s\\
    \leq& C(|t'-t|\wedge 1)\,,
\end{align*}
and
\begin{align*}
    H_1^{(2)}\leq& \left(|b(0)|+\Lip_b \sup_{\alpha\in (\alpha_0,1)}\sup_{s\in [0,t]}\sup_{y\in \R}\left\|u_\alpha(s,y)\right\|_k\right)\int_t^{t'}\int_\R G_{t'-s}(x-y)\ud y \ud s\\
    =&\left(|b(0)|+\Lip_b \sup_{\alpha\in (\alpha_0,1)}\sup_{\alpha\in (\alpha_0,1)}\sup_{s\in [0,t]}\sup_{y\in \R}\left\|u_\alpha(s,y)\right\|_k\right)(t'-t)^2\\
    \leq& C_T(|t'-t|\wedge 1)^2\,.
\end{align*}
For $H_2$,
\begin{align*}
    &H_2\leq \left\|\int_0^{t}\int_\R \left(G_{t'-s}(x-y)-G_{t-s}(x-y)\right)\sigma(u_\alpha(s,y))\eta_\alpha(\ud s,\ud y )\right\|_k\\
    &\quad\quad +\left\|\int_t^{t'}\int_\R G_{t'-s}(x-y)\sigma(u_\alpha(s,y))\eta_\alpha(\ud s,\ud y)\right\|_k\\
    &\quad=H_2^{(1)}+H_2^{(2)}\,.
\end{align*}
We apply Burkholder's inequality, Minkowski inequality and Lemma~\ref{time holder} to get
\begin{align*}
    H_2^{(1)}\leq& C_k\left(|\sigma(0)|+\Lip_\sigma \sup_{\alpha\in (\alpha_0,1)}\sup_{s\in [0,t]}\sup_{y\in \R}\left\|u_\alpha(s,y)\right\|_k\right)\\
    &\times\bigg(\int_0^t\int_\R\int_\R\left|G_{s+(t'-t)}(x-y)-G_{s}(x-y)\right|f_{\alpha}(y-z)\\
    &\times \left|G_{s+(t'-t)}(x-z)-G_{s}(x-z)\right|\ud y\ud z\ud s\bigg)^{\frac{1}{2}}\\
    \leq&  C(|t'-t|\wedge 1)^{\beta}\,.
\end{align*}
and
\begin{align*}
    H_2^{(2)}\leq& C\left(|\sigma(0)|+\Lip_\sigma \sup_{\alpha\in (\alpha_0,1)}\sup_{s\in [0,T]}\sup_{y\in \R}\left\|u_\alpha(s,y)\right\|_k\right)\\
    &\times \left(\int_0^{t'-t}\int_\R\left|\calF(G_s(x-\cdot)(\xi))\right|^2\hat{f}_\alpha(\xi)\ud \xi \ud s\right)^{\frac{1}{2}}\\
    \leq& C\int_0^{t'-t}(s^2\vee 1)\ud s\cdot \int_\R \frac{2\hat{f}_\alpha(\xi)}{1+|\xi|^2}\ud \xi \\
    \leq&  C(|t'-t|^{3}\vee |t'-t|)\,.
\end{align*}
And as a result,
\begin{equation}
    \left\|u_\alpha(t',x)-u_\alpha(t,x)\right\|_k\leq C(|t'-t|\wedge 1)^{\beta}\leq C|t'-t|^{\beta}\,.
\end{equation}
The proof is completed.
\end{proof}
In the remaining part of this paper, we assume $u_\alpha$ is continuous.

\subsection{Finite dimensional convergence}\label{Finite dimensional convergence}
In this section, we prove the convergence of finite dimensional distribution. The following theorem is established.
\begin{theorem}\label{convergence in k-th norm}
Let $u_\alpha$ be the solution to \eqref{colored SWE} for any fixed $\alpha\in I\bigcup\{\alpha_0\}$,
    \begin{equation}\label{equation: convergence in k-th norm}
        \lim_{\alpha\xrightarrow{}\alpha_0}\sup_{t\in [0,T]}\sup_{x\in \R}\|u_\alpha(t,x)-u_{\alpha_0} (t,x)\|_k=0\,.
    \end{equation}
\end{theorem}

We introduce scheme of Picard iterations. The initial iteration is given as
\begin{align*}
    u_\alpha^{(0)}(t,x)=u ^{(0)}(t,x)=I_0(t,x)\,.
\end{align*}
Define recursively,
\begin{align}\label{u alpha picard original}
        \begin{aligned}
            u_\alpha^{(n+1)}(t,x)=&I_0(t,x)+\int_0^t\int_\R G_{t-s}(x-y)b(u_\alpha^{(n)}(s,y))\ud y\ud s\\
        &+\int_0^t\int_\R G_{t-s}(x-y)\sigma(u_\alpha^{(n)}(s,y))\eta_\alpha(\ud s,\ud y)\,,
        \end{aligned}
\end{align}
and especially,
\begin{equation}\label{u picard}
        \begin{aligned}
            u_{\alpha_0} ^{(n+1)}(t,x)=&I_0(t,x)+\int_0^t\int_\R G_{t-s}(x-y)b(u _{\alpha_0}^{(n)}(s,y))\ud y\ud s\\
        &+\int_0^t\int_\R G_{t-s}(x-y)\sigma(u _{\alpha_0}^{(n)}(s,y))\eta_{\alpha_0}(\ud s,\ud y)\,.
        \end{aligned}
\end{equation}
The following two lemmas proves the continuity of these Picard iterations. The proofs basically follow the same structure as Proposition~\ref{holder}.
\begin{lemma}\label{picard alpha property}
    $\|u_\alpha^{(n)}(t,x)\|_k$ is uniformly bounded in all non-negative integer $n$
    \begin{equation}\label{picard uniform bound alpha}
        \sup_{\alpha\in I}\sup_{n\geq 0}\sup_{t\in [0,T]}\sup_{x\in \R}\|u_\alpha^{(n)}(t,x)\|_k<\infty\,,\quad \forall k\geq 2\,.
    \end{equation}
    and $u_\alpha^{(n)}(t,x)$ has a continuous modification.
\end{lemma}
\begin{proof}
    We may see $\calN_{\gamma,k}(u_\alpha^{(n)})$ is uniformly bounded in all the non-negative integers $n$ and gamma large enough  $\gamma>\gamma_1$ where $\gamma_1$ is constant for fixed $k$. Indeed, we start from \eqref{u picard} and follow the same steps as in the proof of Proposition~\ref{uniform boundedness}, and get
\begin{align*}
    \mathcal{N}_{\gamma,k}(u_\alpha^{(n+1)})^2\leq 3C_0+3C_1+3C_2+\frac{3C_1}{\gamma^4}\mathcal{N}_{\gamma,k}(u_\alpha^{(n)})^2+\frac{3C_2}{\gamma^2}\mathcal{N}_{\gamma,k}(u_\alpha^{(n)})^2\,,
\end{align*}
where $C_0$, $C_1$ and $C_2$ are the same as in proof of Proposition~\ref{uniform boundedness}.
 Thus, we can choose $\gamma_1=\max\{(12C_1)^{\frac{1}{4}},(12C_2)^{\frac{1}{2}},2\}$ and we have for all $\gamma>\gamma_1$,
\begin{equation}\label{half calN iteration}
    \mathcal{N}_{\gamma,k}(u_\alpha^{(n+1)})^2\leq C+\frac{1}{2}\mathcal{N}_{\gamma,k}(u_\alpha^{(n)})^2\,.
\end{equation}
Thus,
\begin{align*}
    \mathcal{N}_{\gamma,k}(u_\alpha^{(n)})^2\leq \left(\frac{1}{2}\right)^{n}\mathcal{N}_{\gamma,k}(u_\alpha^{(0)})^2+\sum_{i=0}^{n-1}\left(\frac{1}{2}\right)^{i}C\leq \sup_{y\in \R}|u_0(y)|^2+2C\,,
\end{align*}
uniformly for all $n\in \bbN$ and $\gamma>\gamma_1$. Multiply both sides by $e^{2\gamma T}$ then we get the uniform bound \eqref{picard uniform bound alpha}. The rest of the proof follows immediately with the same steps in proof of Proposition~\ref{holder}, using \eqref{picard uniform bound alpha} instead of \eqref{alpha uniform bound}, which is omitted.
\end{proof}
Thanks to Lemma~\ref{picard alpha property}, we can see that the condition \eqref{coupling condition} is fulfilled when taking $X(s,y)=G_{t-s}(x-y)$. Thus, by the coupling relation \eqref{coupling result},
\begin{equation}\label{u alpha picard}
    \begin{aligned}
        u_\alpha^{(n+1)}(t,x)=&I_0(t,x)+\int_0^t\int_\R G_{t-s}(x-y)b(u_\alpha^{(n)}(s,y))\ud y\ud s\\
        &+\int_0^t\int_\R \left[\left(G_{t-s}(x-\cdot)\sigma(u_\alpha^{(n)}(s,\cdot))\right)\ast h_\alpha\right](y)\eta(\ud s,\ud y)\,.
    \end{aligned}
\end{equation}
where $\eta$ is the white noise defined in Section~\ref{definition of the noise}.
And the same for $\alpha_0$.
Now we are about to prove Theorem~\ref{convergence in k-th norm}.
\begin{proof}[Proof of Theorem~\ref{convergence in k-th norm}]
Since ${\alpha_0}$ is fixed, we write $u:=u_{\alpha_0}$, $h:=h_{\alpha_0}$ and $f:=f_{\alpha_0}$ for simplicity.
Make subtraction of \eqref{u alpha picard} with \eqref{u picard} then take the $k$-th norm,
\begin{align*}
    &\left\|u_\alpha^{(n+1)}(t,x)-u^{(n+1)}(t,x)\right\|_k\\
    \leq &\left\|\int_0^t\int_\R G_{t-s}(x-y)\left(b(u_\alpha^{(n)}(s,y))-b(u ^{(n)}(s,y))\right)\ud y\ud s\right\|_k\\
    & +\left\|\int_0^t\int_\R \left[G_{t-s}(x-\cdot)(\sigma(u_\alpha^{(n)}(s,\cdot))-\sigma(u ^{(n)}(s,\cdot))\ast h_\alpha\right](y)\eta(\ud s,\ud y)\right\|_k\\
    & +\left\|\int_0^t\int_\R\left[G_{t-s}(x-\cdot)\sigma(u ^{(n)}(s,\cdot))\right]\ast (h_\alpha-h )(y)\eta( \ud s,\ud y)\right\|_k\\
     =&I_1+I_2+I_3\,.
\end{align*}
Now we evaluate each term. For $I_1$, multiply $e^{-\gamma t}$ with some $\gamma>1$,
\begin{align*}
    &e^{-\gamma t}I_1\leq\int_0^te^{-\gamma(t-s)}\int_\R e^{-\gamma s}\Lip_b G_{t-s}(x-y)\left\|u_\alpha^{(n)}(s,y)-u ^{(n)}(s,y)\right\|_k\\
    &\qquad\leq \Lip_b \calN_{\gamma,k}(u_\alpha^{(n)}-u ^{(n)})\int_0^te^{-\gamma(t-s)}(t-s)\ud s\\
    &\qquad\leq \frac{\Lip_b}{\gamma}\calN_{\gamma,k}(u_\alpha^{(n)}-u ^{(n)})\,.
\end{align*}
For $I_2$, using Lemma~\ref{fourier mode},
\begin{align*}
    &e^{-\gamma t}I_2\leq C_k\Big\|\int_0^t e^{-2\gamma(t-s)}\int_\R\int_\R e^{-2\gamma s}G_{t-s}(x-y)\left(\sigma(u_\alpha^{(n)}(s,y))-\sigma(u ^{(n)}(s,y))\right)\\
    &\qquad \qquad\times f_\alpha(y-z)G_{t-s}(x-z)\left(\sigma(u_\alpha^{(n)}(s,z))-\sigma(u ^{(n)}(s,z))\right)\ud y \ud z \ud s\Big\|_{\frac{k}{2}}^{\frac{1}{2}}\\
    &\qquad\leq C\Big(\int_0^t e^{-2\gamma(t-s)}\int_\R \int_\R \Lip_\sigma^2 e^{-2\gamma s}\sup_{y\in \R}\|u_\alpha^{(n)}(s,y)-u ^{(n)}(s,y)\|_k^2\\
    &\qquad\qquad \times G_{t-s}(x-y)f_\alpha(y-z)G_{t-s}(x-z)\ud y\ud z\ud s\Big)^{\frac{1}{2}}\\
    &\qquad \leq C\Lip_\sigma \calN_{\gamma,k}(u_\alpha^{(n)}-u ^{(n)})\left(\int_0^t e^{-2\gamma(t-s)}\int_\R \hat{f}_\alpha(\xi)|\calF G_{t-s}(\xi)|^2\ud \xi\ud s\right)^{\frac{1}{2}}\\
    &\qquad = C\Lip_\sigma \calN_{\gamma,k}(u_\alpha^{(n)}-u ^{(n)})\left(\int_0^te^{-2\gamma(t-s)}\int_\R\frac{\hat{f}_\alpha(\xi)}{1+|\xi|^2}(t-s)^2\ud \xi \ud s \right)^{\frac{1}{2}}\\
    &\qquad \leq \left(\frac{C}{\gamma^{3}}\Gamma(3)\right)^{\frac{1}{2}}\Lip_\sigma \calN_{\gamma,k}(u_\alpha^{(n)}-u ^{(n)})\,.
\end{align*}
For $I_3$, apply Burkholder's inequality,
{$$
\begin{aligned}
    I_3=&\bigg\|\int_0^t\int_{\R}\left[G_{t-s}(x-\cdot)\sigma(u ^{(n)}(s,\cdot))*h_{\alpha}\right](y)\eta(\ud s,\ud y) \\
    &\qquad- \int_0^t\int_{\R}[G_{t-s}(x-\cdot)\sigma(u ^{(n)}(s,\cdot))\ast h ](y)\eta(\ud s,\ud y)\bigg\|_k\\
    &\leq C_k \bigg\|\int_0^t\int_{\R}\left[G_{t-s}(x-\cdot)\sigma(u ^{(n)}(s,\cdot))*h_{\alpha}\right]^2(y)\ud y \ud s\\
    &\qquad - 2 \int_0^t\int_{\R}\left[G_{t-s}(x-\cdot)\sigma(u ^{(n)}(s,\cdot))*h_{\alpha}\right](y)\times [G_{t-s}(x-y)\sigma(u ^{(n)}(s,y))\ast h ]\ud y \ud s\\
    &\qquad\quad + \int_0^t\int_{\R}\left[G_{t-s}(x-\cdot)\sigma(u ^{(n)}(s,\cdot))*h \right]^2(y)\ud y \ud s \bigg\|_{\frac{k}{2}}^{1/2}\\
    &= C_k \bigg\|\int_0^t\int_{\R}\int_\R \left(G_{t-s}(x-y)\sigma(u ^{(n)}(s,y))f_\alpha(y-z)G_{t-s}(x-z)\sigma(u ^{(n)}(s,z))\right))\ud z  \ud y\ud s\\
    &\quad - 2 \int_0^t\int_{\R^2}
    \left(G_{t-s}(x-z)\sigma(u ^{(n)}(s,z))(h_\alpha\ast h)(z-w)G_{t-s}(x-w)\sigma(u ^{(n)}(s,w))\right))\ud z\ud w \ud s\\
    &\quad + \int_0^t\int_{\R}\int_\R \left(G_{t-s}(x-y)\sigma(u ^{(n)}(s,y))f (y-z)G_{t-s}(x-z)\sigma(u ^{(n)}(s,z))\right))\ud y\ud z\ud s \bigg\|_{\frac{k}{2}}^{1/2}\\
\end{aligned}
$$}
By Lemma~\ref{picard alpha property}, $u ^{(n)}(t,x)$ is continuous, $G_{t-s}(x-\cdot)\sigma(u^{(n)}(s,\cdot))$ is a bounded function with compact support. We apply Lemma~\ref{fourier mode}, to obtain 

\begin{align*}
    I_3\leq& C\left\|\int_0^t\int_\R (\hat{h}_\alpha^2-2\hat{h}_\alpha\hat{h}+\hat{h}^2)(\xi)\left|\calF(G_{t-s}(x-\cdot)\sigma(u^{(n)}(s,\cdot)))(\xi)\right|^2\ud \xi\ud s\right\|_{\frac{k}{2}}^{\frac{1}{2}}\\
    \leq &C\left\|\int_0^t\int_{[-\omega,\omega]} (\hat{h}_\alpha^2-2\hat{h}_\alpha\hat{h}+\hat{h}^2)(\xi)\left|\calF(G_{t-s}(x-\cdot)\sigma(u^{(n)}(s,\cdot)))(\xi)\right|^2\ud \xi\ud s\right\|_{\frac{k}{2}}^{\frac{1}{2}}\\
    &+C\left\|\int_0^t\int_{\R\setminus {[-\omega,\omega]}} (\hat{h}_\alpha^2-2\hat{h}_\alpha\hat{h}+\hat{h}^2)(\xi)\left|\calF(G_{t-s}(x-\cdot)\sigma(u^{(n)}(s,\cdot)))(\xi)\right|^2\ud \xi\ud s\right\|_{\frac{k}{2}}^{\frac{1}{2}}\\
    =&I_{4}+I_{5}\,,
\end{align*}
where $\omega$ is given in the Assumption~\ref{assumption on f}. 
By the Fourier transform \eqref{fourier transform}, the Fourier part can be estimated as following,
\begin{align*}
    &\left|\calF \left(G_{t-s}(x-\cdot)\sigma(u^{(n)}(s,\cdot))\right)(\xi)\right|^2\\
    &\quad \leq \left\|G_{t-s}(x-\cdot)\sigma(u^{(n)}(s,\cdot))\right\|_{L^1(\R)}^2\\
    &\quad \leq C \left\|G_{t-s}(x-\cdot)(1+u_\alpha^{(n)}(s,\cdot))\right\|^2_{L^1(\R)}\\
    &\quad \leq C\left(2(t-s)^2+2\left\|G_{t-s}(x-\cdot)u^{(n)}(s,\cdot)\right\|_{L^1(\R)}^2\right)\\
    &\quad \leq C\left(T^2+\left\|G_{t-s}(x-\cdot)u^{(n)}(s,\cdot)\right\|_{L^1(\R)}^2\right )\,.
\end{align*}
Thus, we have
\begin{align*}
    &I_4\leq C\left\|\int_0^t\int_{[-\omega,\omega]} |\hat{h}_\alpha^2-2\hat{h}_\alpha\hat{h}+\hat{h}^2|(\xi)\left(T^2+\left\|G_{t-s}(x-\cdot)u^{(n)}(s,\cdot)\right\|_{L^1(\R)}^2\right )\ud \xi\ud s\right\|_{\frac{k}{2}}^{\frac{1}{2}}\\
    &\quad \leq CT^2\left(\int_0^t\int_{[-\omega,\omega]} |\hat{h}_\alpha^2-2\hat{h}_\alpha\hat{h}+\hat{h}^2|(\xi)\ud \xi\ud s\right)^{\frac{1}{2}}\\
    &\quad\quad +C\left(\int_0^t\int _{[-\omega,\omega]}|\hat{h}_\alpha^2-2\hat{h}_\alpha\hat{h}+\hat{h}^2|(\xi)\left\|G_{t-s}(x-\cdot)\right\|_{L^1(\R)}^2\sup_{y\in \R}\|u^{(n)}(s,y)\|_k^2\ud \xi\ud s\right)^{\frac{1}{2}}\\
    &\quad =I_4^{(1)}+I_4^{(2)}\,.
\end{align*}
By triangle inequality, noting that $h_\alpha$ and $h$ are non-negative, we have
\begin{align*}
    I_4^{(1)}=&C_T\int_0\int_{[-\omega, \omega]}|\hat{h}_\alpha-\hat{h}|^2(\xi)\ud \xi \ud s\\
    \leq& C_T\int_0\int_{[-\omega, \omega]}|\hat{h}_\alpha-\hat{h}||\hat{h}_\alpha+\hat{h}|(\xi)\ud \xi \ud s\\
    \leq & C_T\int_0\int_{[-\omega, \omega]}|\hat{f}_\alpha-\hat{f}|(\xi)\ud \xi \ud s
\end{align*}
which converges to 0 since $\hat{f}_\alpha$ converges to $\hat{f}$ in $L^1_{loc}(\R)$ by Assumption~\ref{assumption on f}.
\begin{align*}
    e^{-\gamma t}I_4^{(2)}\leq& C\left(\int_0^t\int _{[-\omega,\omega]}e^{-2\gamma (t-s)}|\hat{h}_\alpha^2-2\hat{h}_\alpha\hat{h}+\hat{h}^2|(\xi)T^2{\sup_{s\in [0,t]}\sup_{y\in \R}e^{-2\gamma s}\|u^{(n)}(s,y)\|_k^2}\ud \xi\ud s\right)^{\frac{1}{2}}\\
    \leq& CT\calN_{\gamma,k}(u^{(n)})\left(\int_0^t\int _{[-\omega,\omega]}|\hat{h}_\alpha^2-2\hat{h}_\alpha\hat{h}+\hat{h}^2|(\xi)\ud \xi\ud s\right)^{\frac{1}{2}}\,.
\end{align*}
By the proof of Lemma~\ref{picard alpha property}, $\calN_{\gamma,k}(u^{(n)})$ is uniformly bounded in all the non-negative integers $n$ and gamma large enough  $\gamma\geq\gamma_1$ where $\gamma_1$ is constant for fixed $k$. Thus, by Dominated convergence theorem, $e^{-\gamma t}I_4^{(2)}$ converges to $0$ as $\alpha$ goes to $\alpha_0$. By the definition of convergence, for any $\epsilon>0$, there exists $\delta^*$, such that $I_4^{(1)}+e^{-\gamma t}I_4^{(2)}<\epsilon/2$ for any $\alpha\in (\alpha_0-\delta^*,\alpha_0+\delta^*)\bigcup I$.
For $I_5$, we separate the integral into two parts,
\begin{align*}
    &I_{5}\leq C\left\|\int_0^t\int_{\R\setminus {[-\omega,\omega]}} (\hat{h}_\alpha^2-2\hat{h}_\alpha\hat{h}+\hat{h}^2)(\xi)\left|\calF(G_{t-s}(x-\cdot)(\sigma(u^{(n)}(s,\cdot))-\sigma(u(s,\cdot)))(\xi))\right|^2\ud \xi\ud s\right\|_{\frac{k}{2}}^{\frac{1}{2}}\\
    &\quad\quad +C\left\|\int_0^t\int_{\R\setminus {[-\omega,\omega]}} (\hat{h}_\alpha^2-2\hat{h}_\alpha\hat{h}+\hat{h}^2)(\xi)\left|\calF(G_{t-s}(x-\cdot)\sigma(u(s,\cdot)))(\xi)\right|^2\ud \xi\ud s\right\|_{\frac{k}{2}}^{\frac{1}{2}}\\
    &\quad =I_{5,n}^{(1)}+I_5^{(2)}\,.
\end{align*}
Note that $|\hat{h}_\alpha^2-2\hat{h}_\alpha\hat{h}+\hat{h}^2|(\xi)\leq (\hat{f}_\alpha+\hat{f})(\xi)$, which is uniformly in $\alpha$ for any {$\xi \in \R \setminus[-\omega,\omega]$} by Assumption~\ref{assumption on f}, and 
\begin{align*}
    &\left\|\int_0^t\int_{\R\setminus {[-\omega,\omega]}} \left|\calF(G_{t-s}(x-\cdot)(\sigma(u^{(n)}(s,\cdot))-\sigma(u(s,\cdot)))(\xi))\right|^2\ud \xi\ud s\right\|_{\frac{k}{2}}^{\frac{1}{2}}\\
    &\quad \leq \left\|\int_0^t\int_\R \left|G_{t-s}(x-y)(\sigma(u^{(n)}(s,y))-\sigma(u(s,y))\right|^2\ud y \ud s\right\|_{\frac{k}{2}}^{\frac{1}{2}}\\
    &\quad \leq \Lip_\sigma \left(\int_0^t\int_\R G_{t-s}^2(x-y)\ud y\ud s \right)^{\frac{1}{2}}\sup_{s\in [0,t]}\sup_{y\in \R}\left\|u^{(n)}(s,y)-u(s,y)\right\|_k\\
    &\quad\leq C_T \sup_{s\in [0,t]}\sup_{y\in \R}\left\|u^{(n)}(s,y)-u(s,y)\right\|_k
\end{align*}
By property of Picard iteration, $I_{5,n}^{(1)}$ tends to $0$ as $n\rightarrow \infty$ uniformly in $\alpha$. For any $\epsilon>0$ there exist $N\in \mathbb{N}$ , $I_{5,n}^{(1)}<\epsilon/4$ for any $\alpha\in I$ and $n>N$. For any $i\leq N$ fixed, $I_{5,i}^{(1)}$ tends to $0$ as $\alpha$ goes to $1$ by DCT. There exists $\delta_i$ such that $I_{5,i}^{(1)}<\epsilon/4$ for any $\alpha\in (\alpha_0-\delta_i,\alpha_0+\delta_i)\bigcup I$. Also by DCT, there exists $\delta_0$ such that $I_5^{(2)}<\epsilon/4$ for any $\alpha\in (\alpha_0-\delta_0,\alpha_0+\delta_0)\bigcup I$. Define $\delta^{**}=\min_{0\leq i\leq N}\delta_i$, we have $I_5\leq \epsilon/2$, uniformly for all $n\in\mathbb{N}$. Take $\delta=\delta^*\wedge\delta^{**}$
For any $\alpha\in (\alpha_0-\delta,\alpha_0+\delta)\bigcup I$, we have
\begin{align*}
    &e^{-\gamma t}\left\|u_\alpha^{(n+1)}(t,x)-u^{(n+1)}(t,x)\right\|_k\\\
    &\quad \leq e^{-\gamma t}I_1+e^{-\gamma t}I_2+I_4^{(1)}+e^{-\gamma t}I_4^{(2)}+I_5\\
    &\quad\leq \frac{C}{\gamma}\calN_{\gamma,k}(u_\alpha^{(n)}-u^{(n)})+\frac{C}{\sqrt{\gamma}}\calN_{\gamma,k}(u_\alpha^{(n)}-u^{(n)})+\epsilon\,.
\end{align*}
We may choose $\gamma$ large enough, for example, $\gamma= \max\{4C,(4C)^2,\gamma_1,2\}$ and take supremum over $t\in [0,T]$ and over $x\in \R$,
\begin{align*}
    \calN_{\gamma,k}(u_\alpha^{(n+1)}-u^{(n+1)})\leq \frac{1}{2}\calN_{\gamma,k}(u_\alpha^{(n)}-u^{(n)})+\epsilon\,,
\end{align*}
and 
\begin{align}
    \calN_{\gamma,k}(u_\alpha^{(n+1)}-u^{(n+1)})\leq \left(\frac{1}{2}\right)^{n+1}\calN_{\gamma,k}(u_\alpha^{(0)}-u^{(0)})+\sum_{i=0}^{n}\left(\frac{1}{2}\right)^{i}\epsilon\,.
\end{align}
Let $n$ goes to infinity we have the convergence in the norm $\calN_{\gamma,k}$. Notice again the factor $e^{-\gamma t}$ is bounded in $[e^{-\gamma T},1]$, the proof is completed.
\end{proof}
Now we can see the above result implies the convergence in finite dimensional law. Choose arbitrary finite pairs $\{(t_i,x_i)\}_{i=1}^l$, and for all fixed $\epsilon_0>0$, by Chebyshev's inequality,
\begin{align*}
    \mathbb{P}(\sum_{i=1}^l (u_\alpha(t_i,x_i)-u(t_i,x_i))^2>\epsilon_0^2)&\leq \sum_{i=1}^l \mathbb{P}(|u_{\alpha}(t_i,x_i)-u(t_i,x_i)|>\epsilon_0)\\
    &\leq \sum_{i=1}^l\frac{\E(u_{\alpha}(t_i,x_i)-u(t_i,x_i))^2}{\epsilon_0^2}\,.
\end{align*}
By taking $\alpha\xrightarrow{}\alpha_0$ and using the equality \eqref{equation: convergence in k-th norm} with $k=2$, we have the finite dimensional convergence in probability, hence, 
\begin{align*}
    \left(u_\alpha(t_1,x_1),u_\alpha(t_2,x_2),\dots, u_\alpha(t_,x_l)\right)\xrightarrow{d}\left(u(t_1,x_1),u(t_2,x_2),\dots ,u(t_l,x_l)\right)\,,
\end{align*}
where the notation $\xrightarrow[]{d}$ means convergence in distribution.

\subsection{Proof of Theorem 1.1}
We need the following theorem to show the tightness of the probability measures induced by $\{u_\alpha\}$, in the space $C([0,T]\times \R)$ equipped with supremum norm over compact sets.
\begin{theorem}\cite[Theorem 2.7]{Giordano2020Bernoulli}
    Let $\{X_\lambda\}_{\lambda\in \Lambda}$ be a family of random functions indexed by the set $\Lambda$ and taking values in the space $C([0,T]\times\R)$, in which we consider the metric of uniform convergence over compact sets. Then, the family $\{X_\lambda\}_{\lambda\in \Lambda}$ is tight if, for any compact set $J\subset \R$, there exist $p',p>0,\delta>2,$ and a constant $C$ such that the following holds for any $t,t'\in [0,T]$ and $x,x'\in J$:
    \begin{enumerate}
        \item $\sup_{\lambda\in\Lambda}\E[|X_\lambda(0,0)|^{p'}]<\infty$\,,
        \item $\sup_{\lambda\in\Lambda}\E[|X_\lambda(t',x')-X_\lambda(t,x)|^{p}]<C\left(|t'-t|+|x'-x|\right)^\delta$\,.
    \end{enumerate}
\end{theorem}
Thanks to Section~\ref{Holder regularity}, we can easily see the conditions for tightness are fulfilled by choosing for example $p'=1$ and $p=4$.
The following theorem is an adoption of Theorem 4.15 in \cite{karatzas.shreve:91:brownian} with a trivial modification.
\begin{theorem}\label{finite dim+tight}
    Let $\{X^{(n)}(t,x)\}_{n=1}^\infty$ be a tight sequence of continuous random fields on $[0,T]\times \R$ with the property that, whenever $\{(t_i,x_i)\}_{i=1}^N\subset [0,T]\times \R$, then the sequence of random vectors $\{X^{(n)}(t_1,x_1),\dots,X^{(n)}(t_N,x_N)\}_{n=1}^\infty$ converges in distribution. Let $P_n$ be the measure induced on $(C([0,\infty)\times \R),\mathscr{B}(C([0,\infty)\times \R)))$ by $X^{(n)}$. Then $\{P_n\}_{n=1}^\infty$ converges weakly to a measure $P$, under which the coordinate mapping process $W_{t,x}(\omega):=\omega(t,x)$ on $C([0,\infty)\times\R)$ satisfies
    \begin{align*}
        \left(X^{(n)}(t_1,x_1),\dots,X^{(n)}(t_N,x_N)\right)\xrightarrow[]{d}\left(W_{t_1,x_1},\dots,W_{t_N,x_N}\right),\forall \{(t_i,x_i)\}_{i=1}^N\subset [0,T]\times \R\,, N\geq 1\,,
    \end{align*}
    where the notation $\xrightarrow[]{d}$ means converge in distribution.
\end{theorem}
So far we have proved the convergence of finite dimensional distribution and tightness. We can also see that each $u_\alpha$ has a continuous modification by the H\"{o}lder continuity in Section~\ref{Holder regularity} and Kolmogorov continuity theorem. As a result of Theorem~\ref{finite dim+tight}, together with the uniqueness of weak convergence, for any sequence $\{\alpha_n\}_{n=1}^\infty$ with limit $1$, $u_{\alpha_{n}}$ converge weakly to $u$. Thus, we have completed the proof of Theorem~\ref{main theorem}.

\section{Validation of the examples}\label{Validation of the examples}
We end this paper by verifying the examples in Example~\ref{examples} case by case.
\begin{enumerate}
    \item First of all, let $f_\alpha$ be given by the Riesz kernels \eqref{Riesz kernels}. The Fourier transform of $f_\alpha$ is then given by
    \begin{align*}
            \hat{f}_\alpha(\xi)=\frac{1}{|\xi|^{1-\alpha}}\to \hat{f}_{\alpha_0}(\xi)\,,\quad \text{as} \quad\alpha\to \alpha_0\neq 0\,.
        \end{align*}
        In the special case, $\hat{f}_\alpha\to 1=\hat{\delta}$ as $\alpha\to 1$. In addition, 
    We have $f_\alpha=h_\alpha\ast h_\alpha$ where $h_\alpha=c_{\frac{1-\alpha}{2}}g_{\frac{1+\alpha}{2}}$ and $\hat{f}_\alpha=\hat{h}_\alpha^2$. For any sequence $\{\alpha_n\}$ with limit $\alpha_0\neq 0$, the point-wise convergence of $\hat{f}_\alpha\to\hat{f}_{\alpha_0}$ holds for any $\xi\neq 0$ with $\hat{f}_\alpha(\xi)\leq 1$ uniformly for all $|\xi|\geq 1$ Let $\alpha_l=\inf_n{\alpha_n}>0$. For any compact set $K\subset \R$,
\begin{align*}
    \int_K|\hat{f}_\alpha-\hat{f}_{\alpha_0}|(\xi)\ud \xi\leq 2\left(\int_{K\bigcap[-1,1]}\frac{1}{|\xi|^{1-\alpha_l}}\ud \xi\,\vee |K|\right)\,.
\end{align*}
Thus, by Dominated Convergence Theorem (DCT), $\hat{f}_\alpha\to\hat{f}_{\alpha_0}$ in $L^1_{loc}(\R)$.
Take $\beta=\frac{1}{4}$, we have
\begin{align*}
    \sup_{\alpha}\int_\R\frac{\hat{f}_\alpha(\xi)}{(1+|\xi|^2)^{{1-\beta}}}\ud \xi \leq \int_{[-1,1]}\frac{1}{|\xi|^{1-\alpha_l}}\ud \xi +\int_{\R\setminus[-1,1]}\frac{1}{(|\xi|^2)^{\frac{3}{4}}}\ud \xi< \infty\,,
\end{align*}
which verifies \eqref{improved dalang's condition}.
\item Let $f_\alpha$ be given by \eqref{Heat kernels} and $\hat{f}_\alpha(\xi)=e^{-\alpha|\xi|^2}\to 1=\hat{\delta}(\xi)$ as $\alpha$ goes to 0. We have  $f_\alpha=h_\alpha\ast h_\alpha $ with $h_\alpha=f_{\frac{\alpha}{2}}$. Also, we have the uniform bound $\hat{f}_\alpha(\xi)\leq 1$ and point-wise convergence of $\hat{f}_\alpha\to \hat{f}_{\alpha_0}$ for any $\alpha_0>0$. Noting that
\begin{align*}
    \int_K|\hat{f}_\alpha-\hat{f}_{\alpha_0}|(\xi)\ud \xi\leq 2|K|\,.
\end{align*}
The convergence $\hat{f}_\alpha$ to $\hat{f}_{\alpha_0}$ holds in $L^1_{loc}(\R)$ by DCT.
Again, by noting that $\hat{f}_\alpha(\xi)\leq 1$,take $\beta=\frac{1}{4}$ and we have
\begin{align*}
    \sup_{\alpha}\int_\R\frac{\hat{f}_\alpha(\xi)}{(1+|\xi|^2)^{{1-\beta}}}\ud \xi \leq\int_{[-1,1]}1\ud \xi+ \int_{\R\setminus[-1,1]} \frac{1}{|\xi|^{\frac{3}{2}}}\ud \xi <\infty\,,
\end{align*}
which verifies \eqref{improved dalang's condition}.
\item By the semigroup property, $f_\alpha=f_{\alpha/2}\ast f_{\alpha/2}$. The point-wise convergence and uniform boundedness of $\hat{f}_\alpha(\xi)=\frac{1}{(1+|\xi|^2)^{\frac{\alpha}{2}}}$ on $\R$ are trivial. Also, by DCT, we have the convergence of $\hat{f}_\alpha\to \hat{f}_{\alpha_0}$ in $L^1_{loc}(\R)$ for any $\alpha_0>0$. Again, by noting that  $\hat{G}_\alpha(\xi)\leq 1$, we can show \eqref{improved dalang's condition} with the same step as in the second example.
\item Let $f_\alpha$ be given in \eqref{general example}. We can still prove the uniform boundedness of $\hat{f}_\alpha$ on $\R$ by noting
        \begin{align*}
            \hat{f}_\alpha(\xi)=\left|\int_\R \rho(x)e^{-ia\xi x}\ud x\right|^2\leq \|\rho\|_{L^1(\R)}^2=1\,.
        \end{align*}
        And the improved Dalang's condition \eqref{improved dalang's condition} and the $L^1_{loc}(\R)$ convergence follow immediately.
        Since $\rho\in L^1(\R)$, by DCT, for any $\xi\in\R$, we have the point-wise convergence
        \begin{align*}
            \hat{f}_\alpha(\xi)=\left|\int_\R \rho(x)e^{-ia\xi x}\ud x\right|^2\to\left|\int_\R \rho(x)e^{-ia_0\xi x}\ud x\right|^2=\hat{f}_{\alpha_0}(\xi)\,,
        \end{align*}
        as $\alpha$ goes to $\alpha_0$ for any $\alpha_0>0$.
\end{enumerate}
\section{Appendix}\label{Appendix}
In this section we present some basic and detailed calculations about the wave kernel $G$.
\begin{lemma}\label{esitimation of the Fourier wave kernel}
Let $G_t$ be given by \eqref{wave kernel} and thus $  \mathcal{F}G_t(\xi)=\frac{\sin t|\xi|}{|\xi|}$. Then
    \begin{align*}
        \left|\mathcal{F}G_t(\xi)\right|^2\leq \frac{2}{1+|\xi|^2}(t^2\vee 1)\,.
    \end{align*}
\end{lemma}
\begin{proof}
If $\xi\leq 1$, then
    \begin{align*}
        \frac{1+|\xi|^2}{|\xi|^2}\sin^2(t|\xi|)\leq t^2(1+|\xi|^2)\leq 2t^2\,.
    \end{align*}
And if $\xi>1$, we have
    \begin{align*}
        \frac{1+|\xi|^2}{|\xi|^2}\sin^2(t|\xi|)\leq 1+\frac{1}{|\xi|^2}\leq 2\,.
    \end{align*}
\end{proof}
\begin{lemma}\label{spatial holder 1}
Let $G_t(x)$ be the wave kernel \eqref{wave kernel}. For any $x,x'\in \R$ and $t\in [0,T]$, there exists a constant $C$ such that
    \begin{equation}\label{int G-G}
        \int_0^t\int_\R \left|G_s(x-y)-G_s(x'-y)\right|\ud y\ud s\leq  C(|x-x'|\wedge 1).
    \end{equation}
\end{lemma}
\begin{proof}
We may assume $x'\geq x$, by \eqref{wave kernel},
\begin{equation}\label{G-G}
    \left|G_s(x-y)-G_s(x'-y)\right|=\left\{
    \begin{aligned}
        &\frac{1}{2}\left(\mathds{1}_{\{|x-y|<s\}}(y)+\mathds{1}_{\{|x'-y|<s)\}}(y)\right)\,,\quad &&s\leq \frac{|x'-x|}{2}\,,\\
        &\frac{1}{2}\left(\mathds{1}_{\{x-s<y<x'-s\}}(y)+\mathds{1}_{\{x+s<y<x'+s\}}(y)\right)\,,\quad &&s> \frac{|x'-x|}{2}\,.
    \end{aligned}
    \right.
\end{equation}
When $t\leq \frac{|x'-x|}{2}$,
\begin{align*}
    \int_0^t\int_\R \left|G_s(x-y)-G_s(x'-y)\right|\ud y\ud s=\int_0^t 2s \ud s=2t^2\leq t\frac{|x'-x|}{2}\leq C|x'-x|\,.
\end{align*}
When $t>\frac{|x'-x|}{2}$,
\begin{align*}
    \int_0^t\int_\R \left|G_s(x-y)-G_s(x'-y)\right|\ud y\ud s&=\frac{|x-x'|^2}{4}+\int_{\frac{|x'-x|}{2}}^t |x'-x|\ud s\\
    &=|x-x'|\left(t-\frac{|x-x'|}{4}\right)\leq \frac{t}{2}|x-x'|\leq C|x'-x|.
\end{align*}
Also, we observe that the left hand side of \eqref{int G-G} is bounded, since
\begin{align*}
    &\int_0^t\int_\R \left|G_s(x-y)-G_s(x'-y)\right|\ud y\ud s\leq \int_0^t \frac{1}{2}\cdot 2(2s+2s)\ud s\leq 2T^2\leq C\,.
\end{align*}
Thus, \eqref{int G-G} is proved.
    \end{proof}

\begin{lemma}\label{time holder 1}
Let $G_t(x)$ be the wave kernel \eqref{wave kernel}. For any $x\in\R$ and $t,t'\in [0,T]$, there exists a constant $C$ such that
    \begin{equation}
        \int_0^t\int_\R \left|G_{s+|t'-t|}(x-y)-G_s(x-y)\right|\ud y\ud s\leq  C(|t'-t|\wedge 1).
    \end{equation}
\end{lemma}
\begin{proof}
A direct consequence of \eqref{wave kernel} gives
    \begin{align*}
        \int_0^t\int_\R \left|G_{s+|t'-t|}(x-y)-G_s(x-y)\right|\ud y\ud s =\int_0^t |t'-t|\ud s\leq T|t'-t|\leq 2T^2\leq C\,.
    \end{align*}
\end{proof}
\begin{lemma}\label{spatial holder}
Let $G_s(x)$ be the wave kernel \eqref{wave kernel} and $f_\alpha$ be the covariance function. For any $x,x'\in \R$ and $t\in [0,T]$ there exists a constant $C$ such that
    \begin{equation}
        \begin{aligned}
            \mathcal{A}_{\alpha}:=&\int_0^t\int_\R\int_\R \left|G_{s}(x-y)-G_s(x'-y)\right|\\
            &\qquad \times f_\alpha(y-z)\left|G_{s}(x-z)-G_s(x'-z)\right|\ud y \ud z \ud s\leq C(|x'-x|\wedge 1)^{2\beta}\,.
        \end{aligned}
    \end{equation}
\end{lemma}
\begin{lemma}\label{time holder}
Let $G_s(x)$ be the wave kernel \eqref{wave kernel} and $f_\alpha$ be the covariance function. For any $0\leq t\leq t'\leq T$ an $x\in \R$, there exists a constant $C$ such that
    \begin{equation}
        \begin{aligned}
            \mathcal{B}_{\alpha}:=&\int_0^t\int_\R\int_\R \left|G_{s+(t'-t)}(x-y)-G_{s}(x-y)\right|\\
            &\qquad\qquad\times f_{\alpha}(y-z)\left|G_{s+(t'-t)}(x-z)-G_{s}(x-z)\right|\ud y\ud z\ud s\leq C(|t'-t|\wedge 1)^{2\beta}\,.
        \end{aligned}
    \end{equation}
\end{lemma}
\begin{proof}[Proof of Lemma~\ref{spatial holder}]
Apply Lemma~\ref{fourier mode} to $\mathcal{A}_\alpha$,
\begin{align*}
    \mathcal{A}_\alpha=&\int_0^t\int_\R \hat{f}_\alpha(\xi)|\mathcal{F}G_{s}(x-\cdot)(\xi)-\mathcal{F}G_{s}(x'-\cdot)(\xi)|^2\ud \xi \ud s\\
    =&\int_0^t\int_\R \hat{f}_\alpha(\xi)\left| \left(e^{-i\xi x}-e^{-i\xi x'}\right) \left(\frac{\sin(s|\xi|)}{|\xi|}\right)\right|^2\ud \xi \ud s\\
    =&\int_0^t\int_\R \hat{f}_\alpha(\xi)\left| \left(e^{-i\xi (x-x')}-1\right) \left(\frac{\sin(s|\xi|)}{|\xi|}\right)\right|^2\ud \xi \ud s\\
    \leq&4\int_0^t\int_\R \hat{f}_\alpha(\xi) ((|x-x'|^{2\beta} |\xi|^{2\beta})\wedge 1)\frac{2(s^2\vee 1)}{1+|\xi|^2}\ud \xi \ud s\,.
\end{align*}
where the inequality holds since $|e^{ix}-1|\leq |x|\wedge 2\leq 2(|x|\wedge 1)^\beta\leq 2|x|^\beta $.
Then, on the one hand,
\begin{align*}
    \mathcal{A}_\alpha\leq 4\int_0^t 2(s^2\vee 1)\ud s\cdot\int_\R  \frac{\hat{f}_\alpha(\xi)}{1+|\xi|^2}\ud \xi \leq C_T\,.
\end{align*}
On the other hand,
\begin{align*}
    \mathcal{A}_\alpha \leq&4\int_0^t\int_\R \hat{f}_\alpha(\xi) |x-x'|^{2\beta} \frac{|\xi|^{2\beta}}{(1+|\xi|^2)^\beta}\frac{2(s^2\vee 1)}{(1+|\xi|^2)^{1-\beta}}\ud \xi \ud s\\
    \leq & 8|x-x'|^{2\beta}\int_0^t (s^2\vee 1)\ud s\cdot\int_\R \frac{\hat{f}_\alpha(\xi)}{(1+|\xi|^2)^{1-\beta}}\ud \xi\\
    \leq &C_T|x-x'|^{2\beta}\,.
\end{align*}
\end{proof}
We can prove Lemma~\ref{time holder} in the same method.
\begin{proof}[Proof of Lemma~\ref{time holder}]
\begin{align*}
    \mathcal{B}_\alpha=&\int_0^t\int_\R \hat{f}_\alpha(\xi)\left|\mathcal{F}G_{s+(t'-t)}(x-\cdot)(\xi)-\mathcal{F}G_s(x-\cdot)(\xi)\right|^2\ud \xi \ud s\\
    =&\int_0^t\int_\R \hat{f}_\alpha(\xi)\left|\frac{\sin((s+t'-t)|\xi|)}{|\xi|}-\frac{\sin(s|\xi|)}{|\xi|}\right|^2\ud \xi \ud s\\
    \leq & \int_0^t\int_\R \hat{f}_\alpha(\xi)\left|\frac{2\cos\left(\frac{(2s+t'-t)|\xi|}{2}\right)\sin\left(\frac{(t'-t)|\xi|}{2}\right)}{|\xi|}\right|^2\ud \xi \ud s\,.
\end{align*}
On the one hand, by applying Lemma~\ref{esitimation of the Fourier wave kernel},
\begin{equation}\label{B alpha 1}
    \begin{aligned}
        \mathcal{B}_\alpha\leq \int_0^t\int_\R \hat{f}_\alpha(\xi)\frac{2}{1+|\xi|^2}\ud \xi \ud s\leq C\,.
    \end{aligned}
\end{equation}
On the other hand, note that $\sin{x}\leq (x\wedge 1)\leq x^\beta $,
\begin{equation}\label{B alpha 2}
    \begin{aligned}
        \mathcal{B}_\alpha\leq& \int_0^t\int_{|\xi|\leq 1} \frac{\hat{f}_\alpha(\xi)}{(1+|\xi|^2)^{1-\beta}}\frac{(1+|\xi|^2)^{1-\beta}}{2^{1-\beta}}\frac{|t'-t|^2|\xi|^2}{|\xi|^2}2^{1-\beta}\ud \xi \ud s\\
    &+\int_0^t\int_{|\xi|> 1} \frac{\hat{f}_\alpha(\xi)}{(1+|\xi|^2)^{1-\beta}}\frac{1+|\xi|^2}{(1+|\xi|^2)^{\beta}}\frac{|t'-t|^{2\beta}|\xi|^{2\beta}}{{|\xi|^2}}\ud \xi \ud s\\
    \leq & C|t'-t|^2\int_0^t\int_{|\xi|> 1} \frac{\hat{f}_\alpha(\xi)}{(1+|\xi|^2)^{1-\beta}}\ud \xi \ud s\\
    &+C|t'-t|^{2\beta}\int_0^t\int_{|\xi|\leq 1} \frac{\hat{f}_\alpha(\xi)}{(1+|\xi|^2)^{1-\beta}}\left(\frac{(1+|\xi|^2)}{|\xi|^{2}}\right)^{1-\beta}\ud \xi \ud s\\
    \leq &C(|t'-t|^2\vee|t'-t|^{2\beta})\int_0^t\int_\R \frac{\hat{f}_\alpha(\xi)}{(1+|\xi|^2)^{1-\beta}}\ud \xi \ud s\\
    \leq & C_T(|t'-t|^2\vee|t'-t|^{2\beta})\,.
    \end{aligned}
\end{equation}
Therefore, a combination of \eqref{B alpha 1} and \eqref{B alpha 2} shows
\begin{align*}
    \mathcal{B}_\alpha\leq C[(|t'-t|^2\vee|t'-t|^{2\beta})\wedge 1]\leq C(|t'-t|^{2\beta}\wedge 1)\leq C|t'-t|^{2\beta}\,,
\end{align*}
which completes the proof.
\end{proof}
\section*{Acknowledgments}
The author would like to sincerely thank his Ph.D. supervisor, Jingyu Huang, for his valuable feedback and insightful comments on this paper. The author also extends his gratitude to the anonymous referee for the careful review of the manuscript and for providing thoughtful suggestions and constructive comments.
 

\printbibliography[title={References}]

\end{document}